\documentclass{amsart}

\usepackage[T1]{fontenc}
\usepackage[latin10]{inputenc}
\usepackage{amssymb, amsmath, amsthm}
\usepackage{hyperref, amsrefs}
\usepackage{enumerate}
\usepackage{verbatim}

\def\dd{\mathcal D}
\def\sd{\sum_{I \in \dd}}
\def\fd{\mathbf f}
\def\Fd{\mathbf F}
\def\gd{\mathbf g}
\def\Gd{\mathbf G}
\def\fb{\mathcal B}

\def\lx{L^p(\mathbb{R};X)}
\def\lxs{L^{p'}(\mathbb{R};X^{\!*})}

\newtheorem{theorem}{Theorem}[section]

\newtheorem{lemma}[theorem]{Lemma}
\newtheorem{corollary}[theorem]{Corollary}
\newtheorem{remark}[theorem]{Remark}

\title{Linear bounds for Calder\'{o}n-Zygmund operators with even kernel on UMD spaces}

\author{Sandra Pott}
\address{Centre for Mathematical Sciences, University of Lund, P.O. Box 118, SE-221 00 Lund, Sweden}
\email{sandra@maths.lth.se}
\author{Andrei Stoica}
\address{Centre for Mathematical Sciences, University of Lund, P.O. Box 118, SE-221 00 Lund, Sweden}
\email{andrei@maths.lth.se}


\begin{document}

\maketitle

\begin{abstract}
It is well-known that several classical results about Calder\'{o}n-Zygmund singular integral operators can be extended to \(X\)-valued functions if and only if the Banach space \(X\) has the UMD property. The dependence of the norm of an \(X\)-valued Calder\'{o}n-Zygmund operator on the UMD constant of the space \(X\) is conjectured to be linear. We prove that this is indeed the case for sufficiently smooth Calder\'{o}n-Zygmund operators with cancellation, associated to an even kernel.

Our method uses the Bellman function technique to obtain the right estimates for the norm of dyadic Haar shift operators. We then apply the representation theorem of
T. Hyt\"{o}nen to extend the result to general Calder\'{o}n-Zygmund operators.

\end{abstract}

\section{Introduction}

The study of vector-valued singular integral operators emerged at the beginning of the 1980's, when D.L. Burkholder \cite{Bu83} and J. Bourgain \cite{Bo86} characterized the equivalence between a probabilistic/geometric property of a Banach space \(X\) and the boundedness of the tensor extension of the Hilbert transform to \(X\)-valued \(L^p\)-functions. This property of the Banach space \(X\) is now known as the UMD (unconditionality of martingale differences) property. It is also well-known (see \cite{Bo86}) that a space \(X\) has the UMD property if and only if any singular integral operator that is bounded on \(L^q(\mathbb{R})\), for some \(q \in (1,\infty)\), can be extended continuously to \(\lx\), for all \(p \in (1,\infty)\). The UMD spaces are therefore the right setting where one can extend classical results from the Calder\'{o}n-Zygmund and Littlewood-Paley theories to vector-valued functions. One such example is T. Figiel's extension of the \(T(1)\)-theorem to UMD space-valued functions (see \cite{Fi90}). The UMD property also became important in connection to Partial Difference Equations, where functions take values in Banach spaces.

In the 1970's, several years before Burkholder's work, R.A. Hunt, B. Muckenhoupt and R.L. Wheeden \cite{HuMuWh73} and R.R. Coifman and C. Fefferman \cite{CoFe74} showed that a Calder\'{o}n-Zygmund singular integral operator is bounded on the weighted space \(L^p(w)\) if and only if the weight \(w\) belongs to the so-called \(A_p\) class. For the last two decades, an important open problem in Harmonic Analysis was to characterize the dependence of the operator norm on the \(A_p\) characteristic, \([w]_{A_p}\), of the weight. For \(p=2\) this dependence was conjectured to be linear in \([w]_{A_2}\); the problem has become known as the \(A_2\) conjecture. The first step was taken by J. Wittwer \cite{Wi00}, who proved the \(A_2\) conjecture for the martingale transform. Using Bellman function techniques, S. Petermichl and A. Volberg \cite{PeVo02} showed the conjecture for the Beurling-Ahlfors transform. It took a few more years until the \(A_2\) conjecture was proved for the Hilbert transform by S. Petermichl (see \cite{Pe07}). The conjecture was finally settled in 2010 by T.P. Hyt\"{o}nen \cite{Hy12a}. The main ingredient in his proof is the pointwise representation of a general Calder\'{o}n-Zygmund operator as a weighted average over an infinite number of randomized dyadic systems of some simpler operators (called dyadic Haar shifts) in such a way that the estimates for the dyadic Haar shifts depend polynomially on the complexity. We made this digression in order to emphasize the similar flavour between our result and the (now) \(A_2\) theorem. The majority of our arguments was inspired by ideas developed for the proof of the \(A_2\) theorem, as it is also the case in \cite{Hy12b}.

One of the open problems in vector-valued Harmonic Analysis is to describe the dependence of the norm of a vector-valued Calder\'{o}n-Zygmund operator on the UMD constant of the underlying Banach space. The best known estimate for general UMD spaces is quadratic in terms of \(\beta_{p,X}\), even for the Hilbert transform. The linear bound for singular integrals in the special case of spaces of the form \(L^p(\mathcal{M}_p)\), where \(\mathcal{M}_p\) is the noncommutative \(L^p\)-space associated to a semifinite von Neumann algebra \(\mathcal{M}\) equipped with a normal, semifinite, faithful trace, was obtained by J. Parcet \cite{Pa09}. In abstract UMD spaces, the linear bound has only been shown for the Beurling-Ahlfors transform and for some special classes of operators with even kernel (see \cite{GeMSSa10,Hy07}). The linear estimate for the vector-valued martingale transform is essentially the definition of the UMD property. It is also interesting to mention that, as was the case with the \(A_2\) theorem, the linear bound for the Beurling-Ahlfors transform has been known for some time, yet a similar estimate for the Hilbert transform is still open.

In this paper we prove the linear bound for sufficiently smooth Calder\'{o}n-Zygmund singular integrals with even kernel in general UMD spaces. Our approach uses Bellman function arguments introduced by S. Treil in \cite{Tr11}, but adapted to the vector-valued setting. Using Hyt\"{o}nen's representation of a Calder\'{o}n-Zygmund operator, it is enough to obtain the right estimate for the dyadic Haar shift operators. Unfortunately, the norm of these operators depends exponentially on their complexity, but this issue is overcome by the extra assumption on the smoothness of the Calder\'{o}n-Zygmund operator. The idea to use martingale transforms to estimate dyadic shifts is not new (see \cite{Fi90}). Since we want to obtain a linear bound in terms of \(\beta_{p,X}\) for the norm of dyadic Haar shifts, we have to use the UMD property only once (unlike, for example, Hyt\"{o}nen's paper \cite{Hy12b}, where this property is used twice, thus obtaining the quadratic bound for arbitrary Calder\'{o}n-Zygmund singular integrals). We will decompose a dyadic Haar shift of complexity \(k\) into \(k\) ``slices'' that can be seen as martingale transforms. The main idea is to linearize the norm of these slices and then use the Bellman function to estimate each summand. In order to do this, we start with a standard dyadic martingale of points from the domain of the Bellman function, where at each point we have two choices with equal probability. We will then modify the martingale, but preserving the initial point and the endpoints, and probabilities. From the starting point, instead of going to the next level in the standard martingale, we move with probabilities \(1/2\) to two new points that are ``far enough'' from the initial point, but also ``almost averages'' of the endpoints. We can still move from these new points to the endpoints, this time using a modified dyadic martingale, where at each point we have two choices with ``almost equal'' probability. This new martingale is constructed in such a way that the probabilities of moving from the starting point to the endpoints are still equal, as in the case of the standard martingale. Although we have used probabilistic terms, the formal proof involving the Bellman function is elementary.

The paper is organized as follows: in Section 2 we recall the necessary definitions and results that we are using. Then we state our main result (Theorem \ref{mainthm}) and show that it is enough to obtain a corresponding estimate for dyadic Haar shift operators, which is the content of Theorem \ref{mainshift}. In Section 3 we use the UMD property to relate the norm of a dyadic Haar shift to an expression that will be controlled by the Bellman function. Section 4 contains the definition of the Bellman function associated to our problem and the description of its properties. In Section 5 we formulate and prove the main technical result of the paper, which is inspired by \cite{Tr11}. Finally, in Section 6, we show how the main estimate from the previous section is used to conclude the proof of Theorem \ref{mainshift}.

\section{Definitions and statement of the main results}

In this section, we recall some well-known notions and results that we are going to use later on.

\subsection{Calder\'{o}n-Zygmund operators}

Let \(\Delta=\{(x,x): x \in \mathbb{R}\}\) be the diagonal of \(\mathbb{R} \times \mathbb{R}\). We say that a function \(K: \mathbb{R} \times \mathbb{R} \setminus \Delta \to \mathbb{C}\) is a standard Calder\'{o}n-Zygmund kernel if there exists \(\delta>0\) such that
\[|K(x,y)| \leq \frac{C}{|x-y|},\]
\[|K(x,y)-K(x,z)|+|K(y,x)-K(z,x)| \leq C \frac{|y-z|}{|x-y|^{1+\delta}},\]
for all \(x,y,z \in \mathbb{R}\) with \(|x-y|>2|y-z|\).

An operator \(T\), defined on the class of step functions (which is dense in \(L^p(\mathbb{R})\) for all \(p \in (1,\infty)\)), is called a Calder\'{o}n-Zygmund operator on \(\mathbb{R}\) associated to \(K\) if it satisfies the kernel representation
\[Tf(x)=\int_{\mathbb{R}} K(x,y)f(y)\, dy, \qquad x \notin {\mathrm{supp}}\,f .\]

\subsection{Bochner-Lebesgue spaces}

Let \(X\) be a (real or complex) Banach space. A function \\ \(f:\mathbb{R} \to X\) is called weakly measurable if the map \(\mathbb{R} \ni t \mapsto \langle f(t),x^* \rangle \) is measurable for all \(x^* \in X^*\). We say that  \(f:\mathbb{R} \to X\) is separably valued if there exists a separable subspace \(X_0\) of \(X\) such that \(f(t) \in X_0\) for almost all \(t \in \mathbb{R}\). A function  \(f:\mathbb{R} \to X\) is called strongly measurable if it is both separably valued and weakly measurable.

For \(p \in [1,\infty)\), the Bochner-Lebesgue space \(\lx\) consists of all strongly measurable functions \(f:\mathbb{R} \to X\) such that
\[\|f\|_{\lx} := \Big(\int_{\mathbb{R}}\|f(t)\|^p_{X}\, \mathrm{d}t \Big)^{1/p} < \infty.\]
If \(f \in C^1_c(\mathbb{R})\) and \(x \in X\), we define the tensor product \(f \otimes x: \mathbb{R} \to X\) by \((f \otimes x)(t):=f(t)x\). The set consisting of finite linear combinations of functions of this type is denoted by \(C^1_c(\mathbb{R}) \otimes X\) and is a dense subspace of \(\lx\) if \(p \in [1,\infty)\). An operator \(T\), initially defined on scalar-valued functions, acts on a function \(f=\sum_{k=1}^n f_k \otimes x_k \in C^1_c(\mathbb{R}) \otimes X\) by
\[Tf(t)=\sum_{k=1}^n Tf_k(t)x_k \in X.\]

It is also worth mentioning that if \(p \in [1,\infty)\) and \(X\) is a reflexive space, the dual of \(\lx\) is \(\lxs\), where \(1/p+1/{p'}=1\).

\subsection{UMD spaces}

The natural setting when working with vector-valued singular integrals is the class of the so-called UMD spaces, which we will now introduce. Let \(X\) be a Banach space. A sequence \(\{d_n\}_{n \geq 1}\) of functions from \(\mathbb{R}\) into \(X\) is a martingale difference sequence if for all \(n \geq 1\) and all functions \(\varphi:X^n \to \mathbb{R}\),
\[\int_{\mathbb{R}} \varphi(d_1,\ldots,d_n)d_{n+1}\, \mathrm{d}t =0.\]
We say that \(X\) is a UMD (unconditional for martingale differences) space if there exists a constant \(C_p\) for one (or equivalently, for all) \(p \in (1,\infty)\) such that
\[\Big\|\sum_{k=1}^n \varepsilon_k d_k\Big\|_{\lx} \leq C_p \Big\|\sum_{k=1}^n d_k\Big\|_{\lx},\]
for all \(n \geq 1\), all \(X\)-valued martingale difference sequences \(\{d_k\}_{k=1}^{n}\), and all choices \(\{\varepsilon_k\}_{k=1}^{n}\) of signs from \(\{\pm 1\}\). For each \(p \in (1,\infty)\), the smallest constant in the previous inequality is denoted by \(\beta_{p,X}\) and is called the UMD\(_p\) constant of \(X\). It is obvious that \(\beta_{p,X} \geq 1\) for all \(p \in (1,\infty)\) and all UMD spaces \(X\). We also recall that UMD spaces are reflexive (see, for example, \cite{RdF86}).

The next theorem is due to Figiel and is the best known estimate for the norm of Calder\'{o}n-Zygmund singular integrals in general UMD spaces.

\begin{theorem} [\protect{Figiel \cite{Fi90}}]
Let \(X\) be a UMD space, \(1<p<\infty\), and \(\beta_{p,X}\) be the \(UMD_p\) constant of \(X\). Let \(T\) be a Calder\'{o}n-Zygmund operator on \(\mathbb{R}\) which satisfies the standard kernel estimates, the weak boundedness property \(|\langle T \chi_I, \chi_I \rangle | \leq C|I|\) for all intervals \(I\), and the vanishing paraproduct conditions \(T(1)=T^*(1)=0\). Then
\[\|T\|_{\lx \to \lx} \leq C \beta^2_{p,X},\]
where \(C\) depends only on the constants in the standard estimates and the weak boundedness property.
\end{theorem}

\subsection{Dyadic setting}

The standard dyadic system in \(\mathbb{R}\) is
\[\dd^0:=\bigcup_{j \in \mathbb{Z}} \dd^0_j, \qquad \dd^0_j:=\{2^{-j}([0,1)+k): k \in \mathbb{Z}\}.\]
Given a binary sequence \(\omega=(\omega_j)_{j \in \mathbb{Z}} \in (\{0,1\})^{\mathbb{Z}}\), a general dyadic system on \(\mathbb{R}\) is defined by
\[\dd^{\omega}:=\bigcup_{j \in \mathbb{Z}} \dd^{\omega}_j, \qquad \dd^{\omega}_j:= \dd^0_j + \sum_{i>j} 2^{-i} \omega_i.\]
When the particular choice of \(\omega\) is not important, we will use the notation \(\dd\) for a generic dyadic system. We equip the set \(\Omega:=(\{0,1\})^{\mathbb{Z}}\) with the canonical product probability measure \(\mathbb{P}_{\Omega}\) which makes the coordinates \(\omega_j\) independent and identically distributed with \(\mathbb{P}_{\Omega}(\omega_j=0)=\mathbb{P}_{\Omega}(\omega_j=1)=1/2\). We denote by \(\mathbb{E}_{\Omega}\) the expectation over the random variables \(\omega_j, j \in \mathbb{Z}\).

For an interval \(I \in \dd\), let \(I^+\) and \(I^-\) be the left and right children of \(I\). The parent of \(I\) will be denoted by \(\tilde{I}\). We will also use the notation \[\dd_n(I):=\{J \in \dd: J \subset I, |J|=2^{-n}|I|\}\]
for the collection of \(n\)-th generation children of \(I\), where \(|J|\) stands for the length of the interval \(J\).

To any interval \(I \in \dd\) there is an associated Haar function defined by
\[h_I=|I|^{-1/2}(\chi_{I^+}-\chi_{I^-}),\]
where \(\chi_I\) is the characteristic function of \(I\).

For an arbitrary dyadic system \(\dd\), the Haar functions form an (unconditional) orthogonal basis of \(L^p(\mathbb{R})\) for \(p \in (1,\infty)\). Hence any function \(f \in L^p(\mathbb{R})\) admits the orthogonal expansion
\[f=\sd \langle f,h_I \rangle h_I.\]
We denote the average of a locally integrable function \(f\) on the interval \(I\) by \(\langle f \rangle_I :=|I|^{-1}\int_I f(t)\, \mathrm{d}t\).

A (cancellative) dyadic Haar shift of parameters \((m,n)\), with \(m,n \in \mathbb{N}_0\), is an operator of the form
\[S f = \sum_{L \in \dd}
\sum_{\substack{
            I \in \mathcal{D}_m(L) \\
            J \in \mathcal{D}_n(L)}}
c^{L}_{I,J} \langle f, h_I \rangle h_J,\]
where \(\left |c^{L}_{I,J}\right| \leq \frac{\sqrt{|I|} \sqrt{|J|}}{|L|} =2^{-(m+n)/2} \) and \(f\) is any locally integrable function. The number \(\max\{m,n\}+1\) is called the complexity of the Haar shift.

One of the main results that we are using is the following representation of a Calder\'{o}n-Zygmund operator in terms of dyadic Haar shifts.

\begin{theorem} [\protect{Hyt\"{o}nen \cite{Hy11}}] \label{repr}
Let \(T\) be a Calder\'{o}n-Zygmund operator on \(\mathbb{R}\) which satisfies the standard kernel estimates, the weak boundedness property \(|\langle T \chi_I, \chi_I \rangle | \leq C|I|\) for all intervals \(I\), and $T(1), T^*(1) \in BMO(\mathbb{R})$. Then it has an expansion, say for \(f,g \in C^1_c(\mathbb{R})\),
\[\langle T f,g \rangle_{L^2(\mathbb{R}),L^2(\mathbb{R})} =
C \cdot \mathbb{E}_{\Omega} \sum_{m,n=0}^{\infty} \tau(m,n) \langle S^{mn}_{\omega} f, g \rangle_{L^2(\mathbb{R}),L^2(\mathbb{R})}
+ \mathbb{E}_{\Omega} \langle   \Pi^\omega_{T(1)} f, g \rangle    +  \mathbb{E}_{\Omega} \langle   (\Pi^\omega_{T^*(1)})^* f, g \rangle  ,\]
where \(C\) is a constant depending only on the constants in the standard estimates of the kernel \(K\) and the weak boundedness property, \(S^{mn}_{\omega}\) is a dyadic Haar shift of parameters \((m,n)\) on the dyadic system \(\mathcal{D}^{\omega}\) and \(\tau(m,n) \lesssim P(\max\{m,n\}) 2^{-\delta \max\{m,n\}} \), with \(P\) a polynomial.
\end{theorem}
Here, $\Pi^\omega$ and $(\Pi^\omega)^*$ are dyadic paraproducts with respect to the grid $\dd^\omega$,
$$
         \Pi^\omega_\phi f = \sum_{I \in \dd^\omega}   h_I \langle \phi, h_I \rangle  \langle f \rangle_I,
$$
for $\phi \in BMO(\mathbb{R})$ and suitable $f$.
\subsection{Schur multipliers}

A key ingredient in the proof of our main result is the notion of Schur multiplier. Let \(\mathcal{M}_n(\mathbb{R})\) be the space of \(n \times n\) real matrices. If \(A=(a_{ij}) \in \mathcal{M}_n(\mathbb{R})\), the Schur multiplier is the bounded operator \(S_A:\mathcal{M}_n(\mathbb{R}) \to \mathcal{M}_n(\mathbb{R})\) that acts on a matrix \(M=(m_{ij})\) by Schur multiplication: \(S_A(M)=(a_{ij}m_{ij})\). The Schur multiplier norm is
\[\|A\|_m:=\sup_{M \in \mathcal{M}_n(\mathbb{R})} \frac{\|S_A(M)\|_{op}}{\|M\|_{op}},\]
where \(\|M\|_{op}\) is the operator norm of the matrix \(M\) on \(\ell^2(\{1,2,\ldots,n\})\). If \(A=(a_{ij}) \in \mathcal{M}_n(\mathbb{R})\) is of the form \(a_{ij}=s_i t_j\), then \(A\) is called a rank one Schur multiplier. It is easy to see that if \(A\) is a rank one Schur multiplier, then \(\|A\|_m \leq \|(s_i)_{i=1}^n\|_{\infty} \|(t_j)_{j=1}^n\|_{\infty} \). A classical result due to A. Grothendieck says that the converse is essentially true (up to a constant called Grothendieck constant).

\begin{theorem} [\protect{\cite[Theorem 1.2]{DaDo07}, \cite[Theorem 3.2]{Pi12}}] \label{grot}
The closure of the convex hull of the rank one Schur multipliers of norm one in the topology of pointwise convergence contains the ball of all Schur multipliers of norm at most \(K_G^{-1}\), where \(K_G\) is a universal constant.
\end{theorem}

Here is our main result:

\begin{theorem}\label{mainthm}
Let \(X\) be a UMD space, \(1<p<\infty\), and \(\beta_{p,X}\) be the \(UMD_p\) constant of \(X\). Let \(K\) be an even standard kernel with smoothness \(\delta>1/2\) and \(T\) be a Calder\'{o}n-Zygmund operator on \(\mathbb{R}\) associated to \(K\). Suppose that \(T\) satisfies the weak boundedness property \(|\langle T \chi_I, \chi_I \rangle | \leq C|I|\) for all intervals \(I\), and the vanishing paraproduct conditions \(T(1)=T^*(1)=0\). Then
\[\|T\|_{\lx \to \lx} \leq C \beta_{p,X},\]
where \(C\) depends only on the constants in the standard estimates and the weak boundedness property.
\end{theorem}

\begin{corollary} If the vanishing paraproduct condition is replaced by $T(1)=T^*(1) \in L^\infty(\mathbb{R})$, then the linear bound also holds.
\end{corollary}
\proof of the Corollary. If $\phi =T(1)=T^*(1) \in L^\infty(\mathbb{R})$ and $f \in C^1_c(\mathbb{R}) \otimes X$, then
$$
     \big(\Pi^\omega_{T(1)} + (\Pi^\omega_{T^*(1)})^* \big) f =  \big(\Pi^\omega_{\phi} + (\Pi^\omega_{\phi})^*\big) f =
         \phi f - \sum_{I \in \dd^\omega}   h_I \langle f, h_I \rangle \langle \phi \rangle_I.
$$
The norm of the first term is bounded by $\| \phi\|_\infty \|f \|_{L^p(\mathbb{R}, X)}$, and the norm of the second term by $  \beta_{p,X}\| \phi\|_\infty \|f \|_{L^p(\mathbb{R}, X)}$.
The linear bound for the remaining part of the operator follows from Theorem \ref{mainthm}.
\qed

In order to prove Theorem \ref{mainthm}, it is enough to show a corresponding result for Haar shift operators and then use the representation theorem of T. Hyt\"{o}nen for the case $T(1) = T^*(1) =0$.

A standard argument shows that we may assume that \(X\) is a real Banach space and the kernel \(K\) is real-valued. Since \(T\) is associated to a real-valued even kernel \(K\), \(T^*\), and hence also \({\rm{Re}}\, T =\frac{1}{2}(T+T^*)\), are associated to \(K\). This means that \(T={\rm{Re}}\, T + M_b\), where \(M_b\) is a pointwise multiplication operator with \(b \in L^{\infty}(\mathbb{R})\). But \(\|M_b\|_{\lx \to \lx}=\|b\|_{L^{\infty}(\mathbb{R})} \leq \|b\|_{L^{\infty}(\mathbb{R})} \beta_{p,X} ,\) so by replacing \(T\) with \({\rm{Re}}\, T\), we may assume that \(T\) is a self-adjoint operator.

Let \(f \in \lx\) and \(g \in \lxs\) be fixed. Since \(C_c^1(\mathbb{R}) \otimes X\) and \(C_c^1(\mathbb{R}) \otimes X^*\) are dense in \(\lx\) and \(\lxs\), respectively, we may assume that \(f=\sum_{i=1}^{M} f_i \otimes x_i\) and \(g=\sum_{j=1}^{N} g_j \otimes y_j\), where \(f_1,\ldots,f_M,g_1,\ldots,g_N \in C_c^1(\mathbb{R}),\, x_1,\ldots,x_M \in X\), and \(y_1,\ldots,y_N \in X^*\). We have
\begin{align*}
\langle Tf, g \rangle_{\lx,\lxs}&  = \sum_{i=1}^{M} \sum_{j=1}^N \langle x_i,y_j \rangle_{X,X^*} \langle T f_i,g_j \rangle_{L^p(\mathbb{R}),L^{p'}(\mathbb{R})}\\
& = \sum_{i=1}^{M} \sum_{j=1}^N \langle x_i,y_j \rangle_{X,X^*} \langle T f_i,g_j \rangle_{L^2(\mathbb{R}),L^2(\mathbb{R})}.
\end{align*}
Using Theorem \ref{repr} (with the same notations), the last inner product can be written as
\[\langle T f_i,g_j \rangle_{L^2(\mathbb{R}),L^2(\mathbb{R})} = C \cdot \mathbb{E}_{\Omega} \sum_{m,n=0}^{\infty} \tau(m,n) \langle S^{mn}_{\omega} f_i, g_j \rangle_{L^2(\mathbb{R}),L^2(\mathbb{R})}.\]

Hence we have the representation
\[\langle Tf, g \rangle_{\lx,\lxs} = C \cdot \mathbb{E}_{\Omega} \sum_{m,n=0}^{\infty} \tau(m,n) \langle S^{mn}_{\omega} f, g \rangle_{\lx,\lxs},\]
and therefore,
\[ \|T\|_{\lx \to \lx} \leq C  \sum_{m,n=0}^{\infty} \tau(m,n) \|S^{mn}_{\omega}\|_{\lx \to \lx}.\]
We will show the estimate \(\|S^{mn}\|_{\lx \to \lx} \lesssim (\max\{m,n\}+1) 2^{\max\{m,n\}/2} \beta_{p,X}\) for all dyadic Haar shifts \(S^{mn}\) of parameters \((m,n)\), which ensures the convergence of the series (since \(\delta>1/2\)) and completes the proof of Theorem \ref{mainthm}. Recall that \(T\) was assumed to be self-adjoint, so we may restrict ourselves to the case of self-adjoint dyadic Haar shift of parameters \((m,n)\). This is the content of the following theorem.

\begin{theorem}\label{mainshift}
Let \(X\) be a UMD space, \(1<p<\infty\), and \(\beta_{p,X}\) be the \(UMD_p\) constant of \(X\). Let \(S\) be a self-adjoint dyadic Haar shift of complexity \(k \geq 1\). Then
\[\|S\|_{\lx \to \lx} \leq c \cdot k 2^{k/2} \beta_{p,X},\]
where c is an absolute, positive constant.
\end{theorem}

\section{Reduction of the proof of Theorem \ref{mainshift}}

For a sequence \(\sigma=\{\sigma_I\}_{I \in \dd},\ \sigma_I= \pm 1,\) we define the martingale transform operator \(T_{\sigma}\) by
\[T_\sigma f = \sd \sigma_I  \langle f,h_I \rangle h_I.\]
The martingale transform is considered a good model for Calder\'{o}n-Zygmund singular integral operators. By the definition of UMD spaces we obtain the uniform boundedness (in \(\sigma\)) of these operators. More precisely, we have
\begin{align*}
\sup_{\sigma} \|T_{\sigma} f \|_{\lx} & = \sup_{\sigma}
\sup_{\substack{
            g \in \lxs \\
            \|g\|_{\lxs}=1}}
\langle T_\sigma f,g \rangle_{\lx,\lxs} \\
& =
 \sup_{\sigma}
\sup_{\substack{
            g \in \lxs \\
            \|g\|_{\lxs}=1}}
\sd \sigma_I \big \langle \langle f,h_I \rangle, \langle g,h_I \rangle \big \rangle_{X,X^*} \\
& =
\sup_{\substack{
            g \in \lxs \\
            \|g\|_{\lxs}=1}}
\sd \Big | \big \langle \langle f,h_I \rangle, \langle g,h_I \rangle \big \rangle_{X,X^*} \Big | \leq \beta_{p,X} \|f\|_{\lx}.
\end{align*}
We can rewrite the estimate of the norm as
\begin{align}\label{umd}
\sd \left | \big \langle \langle f \rangle _{I^+} - \langle f \rangle_{I^-}, \langle g \rangle _{I^+} - \langle g\rangle_{I^-}  \big \rangle_{X,X^{*}} \right | \cdot |I| & = 4\sd \left | \big \langle \langle f,h_I \rangle, \langle g,h_I \rangle \big \rangle_{X,X^{*}} \right |  \nonumber \\
 & \leq 4 \beta_{p,X} \|f\|_{\lx} \|g\|_{\lxs},
\end{align}
for all \(f \in \lx\) and \( g \in \lxs.\)

Since \(S\) is a Haar shift operator of complexity \(k\), it has the form
\[S f = \sum_{L \in \dd}
\sum_{\substack{
            I \in \mathcal{D}_m(L) \\
            J \in \mathcal{D}_n(L)}}
c^{L}_{I,J} \langle f, h_I \rangle h_J,\]
where \(\left |c^{L}_{I,J}\right| \leq \frac{\sqrt{|I|} \sqrt{|J|}}{|L|} =2^{-(m+n)/2} \).

For \(0 \leq j \leq k-1\) we introduce the notation \(\mathcal{L} _j :=  \{ I \in \mathcal{D} : |I|=2^{j+kt},  t \in \mathbb{Z} \},\) and define the slice \(S_j\) by \[S_j f = \sum_{L \in \mathcal{L}_j}
\sum_{\substack{
            I \in \mathcal{D}_m(L) \\
            J \in \mathcal{D}_n(L)}}
c^{L}_{I,J} \langle f, h_I \rangle h_J, \]
where \(S_j\) also acts on locally integrable functions. We can thus decompose \(S\) as \(S=\sum _{j=0} ^{k-1} S_j\); the operators \(S_j\) can be seen as martingale transforms when we are moving \(k\) units of time at once, so it is possible to apply the Bellman function method. Notice that these slices are also self-adjoint operators.

Let \(f \in \lx,\  g \in \lxs\) and \(0 \leq j \leq k-1\) be fixed. We can write
\begin{align*}
& \left \langle S_j f,g \right \rangle_{\lx,\lxs}  = \Bigg \langle \sum_{L \in \mathcal{L}_j}
\sum_{\substack{
            I \in \mathcal{D}_m(L) \\
            J \in \mathcal{D}_n(L)}}
c^{L}_{I,J} \langle f, h_I \rangle h_J, \sum_{I' \in \dd} \langle g, h_{I'} \rangle h_{I'} \Bigg \rangle _{\lx,\lxs}\\
& \qquad = \sum_{L \in  \mathcal{L}_j} \sum_{I' \in \dd}
\sum_{\substack{
            I \in \mathcal{D}_m(L) \\
            J \in \mathcal{D}_n(L)}}
c^{L}_{I,J} \big \langle \langle f, h_I \rangle , \langle g, h_{I'} \rangle \big \rangle _{X,X^*} \langle h_J,h_{I'} \rangle_{L^p(\mathbb{R}),L^{p'}(\mathbb{R})} \\
& \qquad = \sum_{L \in  \mathcal{L}_j}
\sum_{\substack{
            I \in \mathcal{D}_m(L) \\
            J \in \mathcal{D}_n(L)}}
c^{L}_{I,J} \big \langle \langle f, h_I \rangle , \langle g, h_J \rangle \big \rangle _{X,X^*}\\
& \qquad = \sum_{L \in  \mathcal{L}_j}
\sum_{\substack{
            I \in \mathcal{D}_m(L) \\
            J \in \mathcal{D}_n(L)}}
c^{L}_{I,J} \frac{|I|^{1/2}}{2^{k-m}} \frac{|J|^{1/2}}{2^{k-n}} \Bigg \langle
\sum_{\substack{
            P \in \mathcal{D}_k(L) \\
            P \subset I^{+}}}
\big ( \langle f \rangle _P -  \langle f \rangle _L \big ) +
\sum_{\substack{
            P \in \mathcal{D}_k(L) \\
            P \subset I^{-}}}
\big ( \langle f \rangle _L -  \langle f \rangle _P \big ) , \Bigg. \\
& \hspace{5 cm} \Bigg.
\sum_{\substack{
            Q \in \mathcal{D}_k(L) \\
            Q \subset J^{+}}}
\big( \langle g \rangle _Q -  \langle g \rangle _L \big ) +
\sum_{\substack{
            Q \in \mathcal{D}_k(L) \\
            Q \subset J^{-}}}
\big ( \langle g \rangle _L -  \langle g \rangle _Q \big )
\Bigg \rangle_{X,X^*}.
\end{align*}

Since the slice \(S_j\) is self-adjoint, we have
\begin{align*}
2 \left \langle S_j f,g \right \rangle_{\lx,\lxs} & = \left \langle S_j f,g \right \rangle_{\lx,\lxs} + \left \langle f,S_j g \right \rangle_{\lx,\lxs}  \\
& = \sum_{L \in  \mathcal{L}_j}
\sum_{\substack{
            I \in \mathcal{D}_m(L) \\
            J \in \mathcal{D}_n(L)}}
c^{L}_{I,J} \Big (\big \langle \langle f, h_I \rangle , \langle g, h_J \rangle \big \rangle _{X,X^*} +  \big \langle \langle f, h_J \rangle , \langle g, h_I \rangle \big \rangle _{X,X^*} \Big ),
\end{align*}
therefore
\begin{align*}
& 2 \Big | \left \langle S_j f,g \right \rangle_{\lx,\lxs} \Big| = \Big | \left \langle S_j f,g \right \rangle_{\lx,\lxs} + \left \langle f,S_j g \right \rangle_{\lx,\lxs} \Big |  \\
& \leq \sum_{L \in  \mathcal{L}_j} |L|
\sum_{P,Q \in \mathcal{D}_k(L)} \bigg | \bigg \langle  \frac{\langle f \rangle _P -  \langle f \rangle _L}{2^k}, \frac{\langle g \rangle _Q -  \langle g \rangle _L}{2^k}    \bigg \rangle_{X,X^*}  +  \bigg \langle  \frac{\langle f \rangle _Q -  \langle f \rangle _L}{2^k}, \frac{\langle g \rangle _P -  \langle g \rangle _L}{2^k}    \bigg \rangle_{X,X^*}   \bigg|.
\end{align*}

\section{The Bellman function}

We are now going to define the Bellman function associated to our problem. Fix a dyadic interval \(I_0\) and for \(\fd \in X, \Fd \in \mathbb{R}, \gd \in X^*, \Gd \in \mathbb{R}\) satisfying
\begin{equation}\label{domain}
\|\fd\|_X^p \leq \Fd,\ \|\gd\|^{p'}_{X^*} \leq \Gd,
\end{equation}
define the function \(\fb=\fb^{I_0}: X \times \mathbb{R} \times X^* \times \mathbb{R} \to \mathbb{R},\) by
\[\fb(\fd, \Fd, \gd, \Gd):= |I_0|^{-1} \sup \sum_{I \subseteq I_0} \left | \big \langle \langle f \rangle _{I^+} - \langle f \rangle_{I^-}, \langle g \rangle _{I^+} - \langle g\rangle_{I^-}  \big \rangle_{X,X^*} \right | \cdot |I|, \]
where the supremum is taken over all strongly measurable functions \(f:\mathbb{R} \to X\) and \(g:\mathbb{R} \to X^*\) such that
\begin{equation}\label{supdomain}
\langle f \rangle _{I_0}=\fd \in X,\quad \big \langle \|f\|^p_X \big \rangle_{I_0} = \Fd \in \mathbb{R}, \quad \langle g \rangle _{I_0}=\gd \in X^*,\quad \big \langle \|g\|^{p'}_{X^*} \big \rangle_{I_0} = \Gd \in \mathbb{R}.
\end{equation}

The Bellman function \(\fb\) has the following properties:
\begin{enumerate}[(i)]
    \item (Domain) The domain \(\mathfrak{D}:=\mathrm{Dom}\, \fb\) is given by \eqref{domain}. This means that for every quadruple \((\fd, \Fd, \gd, \Gd)\) that satisfies \eqref{domain}, there exist functions \(f\) and \(g\) such that \eqref{supdomain} holds, so the supremum is not \(-\infty\). Conversely, if the variables \(\fd, \Fd, \gd, \Gd\) are the corresponding averages of some functions \(f\) and \(g\), then they must satisfy condition \eqref{domain}. It is important to notice that \(\mathfrak{D}\) is a convex set.
    \item (Range) \(0 \leq \fb(\fd, \Fd, \gd, \Gd) \leq 4 \beta_{p,X} \Fd^{1/p} \Gd^{1/{p'}}\) for all \((\fd, \Fd, \gd, \Gd) \in \mathfrak{D}.\)
    \item (Concavity condition) Consider all quadruples \(A=(\fd, \Fd, \gd, \Gd), A_+=(\fd_+, \Fd_+, \gd_+, \Gd_+)\) and \(A_-=(\fd_-, \Fd_-, \gd_-, \Gd_-)\) in \(\mathfrak{D}\) such that \(A=(A_+ + A_-)/2\). For all such quadruples, we have the following concavity condition:
        \[\fb(A) \geq \frac{\fb(A_+)+\fb(A_-)}{2} + \big | \left \langle \fd_+ - \fd_- , \gd_+ - \gd_- \right \rangle_{X,X^*} \big |.\]
  \end{enumerate}
\par
Let us now explain these properties of the function \(\fb\). The inequalities \(\|\fd\|_X^p \leq \Fd\) and \( \|\gd\|^{p'}_{X^*} \leq \Gd\) follow easily from a property of Bochner integrals and H\"{o}lder's Inequality. On the other hand, given a quadruple \((\fd, \Fd, \gd, \Gd) \in \mathfrak{D}\), we set \(f=\fd \varphi\) and \(g=\gd \psi\), where \(\langle \varphi \rangle _{I_0} = \langle \psi \rangle _{I_0} = 1\), \(\big \langle |\varphi|^p \|\fd\|_X^p \big \rangle_{I_0}=\Fd\),  \(\big \langle |\psi|^{p'} \|\gd\|_{X^*}^{p'} \big \rangle_{I_0}=\Gd\), so \(f\) and \(g\) satisfy \eqref{supdomain}.

Property (ii) follows from the definition of \(\fb\) and the inequality \eqref{umd}.

To prove the concavity condition, we consider three quadruples \(A, A_+, A_- \in \mathfrak{D}\) such that \(A=(A_+ + A_-)/2\) and choose two functions \(f\) and \(g\) so that
\begin{equation}\label{avI0+-}
A_{\pm}= \Big ( \langle f \rangle _{I_0^{\pm}},\ \big \langle \|f\|_X^p \big \rangle_{I_0^{\pm}},\ \langle g \rangle _{I_0^{\pm}},\ \big \langle \|g\|_{X^*}^{p'} \big \rangle_{I_0^{\pm}} \Big ).
\end{equation}
Then
\[A=\frac{A_+ + A_-}{2} = \Big ( \langle f \rangle _{I_0},\ \big \langle \|f\|_X^p \big \rangle_{I_0},\ \langle g \rangle _{I_0},\ \big \langle \|g\|_{X^*}^{p'} \big \rangle_{I_0} \Big )\]
is the vector of corresponding averages over \(I_0\). The expression in the definition of \(\fb(\fd, \Fd, \gd, \Gd)\), before taking the supremum, can be split into the average of the corresponding expressions for \(\fb(\fd_+, \Fd_+, \gd_+, \Gd_+)\) and \(\fb(\fd_-, \Fd_-, \gd_-, \Gd_-)\), plus the term \( \big | \big \langle \langle f \rangle _{I_0^+} - \langle f \rangle _{I_0^-}, \langle g \rangle _{I_0^+} - \langle g \rangle _{I_0^-} \big \rangle_{X,X^*} \big |\). Taking now the supremum over all functions \(f\) and \(g\) that satisfy conditions \eqref{avI0+-} we conclude that
\[\frac{\fb(A_+)+\fb(A_-)}{2} + \big | \left \langle \fd_+ - \fd_- , \gd_+ - \gd_- \right \rangle_{X,X^*} \big | \leq \fb(A).\]

This inequality is true because the set of functions over which we are taking the supremum is smaller than the one corresponding to \(\fb(A)\), since we are excluding all those functions \(f\) and \(g\) whose averages on the children of \(I_0\) are not the prescribed values in \eqref{avI0+-}.
\begin{remark}
The concavity condition (iii) implies that the function \(\fb\) is midpoint concave, that is, \(\fb \big(\frac{A_+ + A_-}{2}\big) \geq \frac{1}{2} \big(\fb(A_+) + \fb(A_-) \big)\), for all \(A_+, A_- \in \mathfrak{D}\). It is well-known that locally bounded below midpoint concave functions are actually concave (see \cite{RoVa73}, Theorem C, p. 215). Therefore \(\fb\) is a concave function.
\end{remark}

\section{The main estimate}

The following result is the main tool in the proof of Theorem \ref{mainshift}.
\begin{lemma}\label{mainlemma}
Let \(\fb\) be a function satisfying properties (i)-(iii) from Section 4. Fix \(k \geq 1\) and a dyadic interval \(I_0\). For all \(I \in \mathcal{D}_n(I_0),\ 0 \leq n \leq k,\) let the points \(A_I= (\fd_I, \Fd_I, \gd_I, \Gd_I) \in \mathfrak{D}=\mathrm{Dom}\, \fb\) be given. Assume that the points \(A_I\) satisfy the dyadic martingale dynamics, i.e. \(A_I=(A_{I^+}+A_{I^-})/2,\) where \(I^+\) and \(I^-\) are the children of \(I\). For \(K,L \in \mathcal{D}_k(I_0)\), we define the coefficients \(\lambda_{KL}\) by
\[\lambda_{KL}:= \bigg \langle  \frac{\fd_K - \fd_{I_0}}{2^k}, \frac{\gd_L - \gd_{I_0}}{2^k} \bigg \rangle_{X,X^*}  +  \bigg \langle  \frac{\fd_L - \fd_{I_0}}{2^k}, \frac{\gd_K - \gd_{I_0}}{2^k} \bigg \rangle_{X,X^*}.\]

Then
\[\sum_{K,L \in \mathcal{D}_k({I_0})} |\lambda_{KL}| \leq c \cdot 2^{k/2} \bigg( \fb (A_{I_0}) - 2^{-k} \sum_{I \in \mathcal{D}_k(I_0)} \fb(A_I) \bigg),\]
where \(c\) is a positive absolute constant.
\end{lemma}

\begin{proof}

We introduce the notation
\(\Lambda := \left (\lambda_{KL} \right )_{K,L \in \mathcal{D}_k(I_0)}.\) Assume for the moment that we can find a sequence \((\alpha_I)_{I \in \mathcal{D}_k(I_0)}\) such that \(|\alpha_I| \leq 1/4\) for all \(I \in \mathcal{D}_k(I_0), \sum_{I \in \mathcal{D}_k(I_0)} \alpha_I =0,\) and
\begin{equation}\label{sequence}
\bigg | \sum_{K,L \in \mathcal{D}_k({I_0})} \alpha_K \alpha_L \lambda_{KL} \bigg | \geq c \cdot 2^{-k/2} \sum_{K,L \in \mathcal{D}_k({I_0})} |\lambda_{KL}|.
\end{equation}

Define \(A_{I_0}^{\pm}=(\fd^{\pm}, \Fd^{\pm}, \gd^{\pm}, \Gd^{\pm})\) by
\[A_{I_0}^{\pm}:=2^{-k} \sum_{I \in \mathcal{D}_k(I_0)} (1 \pm \alpha_I)A_I = A_{I_0} \pm 2^{-k} \sum_{I \in \mathcal{D}_k(I_0)} \alpha_I A_I,\]
so \(A_{I_0}=(A_{I_0}^+ + A_{I_0}^-)/2.\) By \eqref{sequence} and the fact that \(\sum_{I \in \mathcal{D}_k(I_0)} \alpha_I =0,\) we have the estimate
\begin{align}\label{est1}
\Big | \big \langle \fd^{\pm} - \fd_{I_0} , \gd^{\pm} - \gd_{I_0} \big \rangle_{X,X^*} \Big | & = \bigg | \bigg \langle 2^{-k} \sum_{I \in \mathcal{D}_k(I_0)} \alpha_I (\fd_I - \fd_{I_0}),  2^{-k} \sum_{I \in \mathcal{D}_k(I_0)} \alpha_I (\gd_I - \gd_{I_0}) \bigg \rangle_{X,X^*} \bigg | \\
& = \bigg | \sum_{K,L \in \mathcal{D}_k({I_0})} \alpha_K \alpha_L \bigg \langle  \frac{\fd_K - \fd_{I_0}}{2^k}, \frac{\gd_L - \gd_{I_0}}{2^k} \bigg \rangle_{X,X^*}  \bigg | \nonumber \\
& =\frac{1}{2} \bigg | \sum_{K,L \in \mathcal{D}_k({I_0})} \alpha_K \alpha_L \lambda_{KL} \bigg | \nonumber \\
& \geq c \cdot 2^{-k/2} \sum_{K,L \in \mathcal{D}_k({I_0})} |\lambda_{KL}|. \nonumber
\end{align}

Let \(a_I^{\pm}:=1 \pm \alpha_I\) and note that \(3/4 \leq a_I^{\pm} \leq 5/4.\)

For \(I \in \mathcal{D}_n(I_0),\ 1 \leq n \leq k,\) let us define
\[A_I^{\pm}:= \left (
\sum_{\substack{
            J \in \mathcal{D}_k(I_0) \\
            J \subseteq I}}
a_J^{\pm}A_J \right ) \left (
\sum_{\substack{
            J \in \mathcal{D}_k(I_0) \\
            J \subseteq I}}
a_J^{\pm} \right )^{-1}.\]
Since the domain of the function \(\fb\) is convex, all these points are in \(\mathfrak{D}\). When \(I \in \mathcal{D}_k(I_0)\) we have \(A_I^+=A_I^-=A_I\), where the \(A_I\)'s are the points from the statement of the lemma.
For \(I \in \mathcal{D}_n(I_0),\ 1 \leq n \leq k,\) we define
\[\theta_I^{\pm}:= \left (
\sum_{\substack{
            J \in \mathcal{D}_k(I_0) \\
            J \subseteq I}}
a_J^{\pm} \right ) \left (
\sum_{\substack{
            J \in \mathcal{D}_k(I_0) \\
            J \subseteq \tilde{I}}}
a_J^{\pm} \right )^{-1}.\]

It is easy to see that \(3/10 \leq \theta_I^{\pm} \leq 5/6\) and
\begin{equation}\label{convexity}
\theta_{I^+}^{\pm} + \theta_{I^-}^{\pm} =1, \qquad A_I^{\pm}=\theta_{I^+}^{\pm} A_{I^+}^{\pm} + \theta_{I^-}^{\pm} A_{I^-}^{\pm}.
\end{equation}

The last equality means that the point \(A_I^+\) is on the line segment with endpoints \(A_{I^+}^+\) and \(A_{I^-}^+\), and similarly for \(A_I^-\). \(\theta_{I^+}^{\pm}\) and \(\theta_{I^-}^{\pm}\) represent the probabilities of moving from the points \(A_I^{\pm}\) to \(A_{I^+}^{\pm}\) and \(A_{I^-}^{\pm}\), respectively.

Since
\begin{align*}
\Big | \big \langle \fd^+ - \fd^- , \gd^+ - \gd^- \big \rangle_{X,X^*} \Big | & = \bigg| \bigg \langle 2 \cdot 2^{-k} \sum_{I \in \mathcal{D}_k(I_0)} \alpha_I \fd_I, 2 \cdot 2^{-k} \sum_{I \in \mathcal{D}_k(I_0)} \alpha_I \gd_I \bigg \rangle_{X,X^*}  \bigg| \\
& = 4 \Big | \big \langle \fd^{\pm} - \fd_{I_0} , \gd^{\pm} - \gd_{I_0} \big \rangle_{X,X^*} \Big |,
\end{align*}
by property (iii) of the Bellman function \(\fb\) we get
\begin{equation}\label{bellmandif}
\big | \left \langle \fd^{\pm} - \fd_{I_0} , \gd^{\pm} - \gd_{I_0} \right \rangle_{X,X^*} \big | \leq \frac{1}{4} \bigg ( \fb(A_{I_0}) - \frac{\fb(A_{I_0}^+)+\fb(A_{I_0}^-)}{2} \bigg ).
\end{equation}

From the concavity of the function \(\fb\) and \eqref{convexity} it follows that
\[\fb(A_I^{\pm}) \geq \theta_{I^+}^{\pm} \fb(A_{I^+}^{\pm}) + \theta_{I^-}^{\pm} \fb(A_{I^-}^{\pm}).\]

Applying now this inequality to \(I \in \mathcal{D}_n(I_0),\ 0 \leq n \leq k-1\), and taking into account that
\[\prod_{\substack{
            J \in \mathcal{D} \\
            I \subseteq J \subsetneq I_0 }}
\theta_J^{\pm} = a_I^{\pm} \bigg (\sum_{J \in \mathcal{D}_k(I_0)} a_J^{\pm} \bigg )^{-1} = 2^{-k} a_I^{\pm}\]
for all \(I \in \mathcal{D}_k(I_0),\) we obtain the estimate
\[\fb(A_{I_0}^{\pm}) \geq 2^{-k} \sum_{I \in \mathcal{D}_k(I_0)} a_I^{\pm} \fb(A_I).\]

Since \(a_I^+ + a_I^-=2\) when \(I \in \mathcal{D}_k(I_0),\) substituting the previous inequality in \eqref{bellmandif} gives
\[\Big | \big \langle \fd^{\pm} - \fd_{I_0} , \gd^{\pm} - \gd_{I_0} \big \rangle_{X,X^*} \Big | \leq \frac{1}{4} \bigg ( \fb(A_{I_0}) - 2^{-k} \sum_{I \in \mathcal{D}_k(I_0)}  \fb(A_I) \bigg ) .\]
Together with \eqref{est1}, this completes the proof of Lemma \ref{mainlemma} under the assumption \eqref{sequence}, which we are now going to prove.

Notice that inequality \eqref{bellmandif} does not depend on property \eqref{sequence} of the sequence \((\alpha_I)_{I \in \mathcal{D}_k(I_0)}.\) Using \eqref{bellmandif}, the equalities from \eqref{est1} and property (ii) of the Bellman function, we have
\[ \bigg | \sum_{K,L \in \mathcal{D}_k({I_0})} \alpha_K \alpha_L \lambda_{KL} \bigg | \leq \frac{1}{2} \fb(A_{I_0}) \leq 2 \beta_{p,X} \Fd_{I_0}^{1/p} \Gd_{I_0}^{1/{p'}},\]
for all sequences \((\alpha_I)_{I \in \mathcal{D}_k(I_0)}\) with \(|\alpha_I| \leq 1/4\) for all \( I \in \mathcal{D}_k(I_0),\) and  \(\sum_{I \in \mathcal{D}_k(I_0)} \alpha_I =0.\)

Let us define
\[\|\Lambda\|_1 := \sup_{\alpha} \bigg | \sum_{K,L \in \mathcal{D}_k({I_0})} \alpha_K \alpha_L \lambda_{KL} \bigg|,\]
where the supremum is taken over all sequences \(\alpha=(\alpha_I)_{I \in \mathcal{D}_k(I_0)}\) with \(\|\alpha\|_{\infty} \leq 1/4\) and \(\sum_{I \in \mathcal{D}_k(I_0)} \alpha_I =0.\) Since we are in a finite-dimensional space, we can find a sequence \(\alpha\) with \(|\alpha_I| \leq 1/4,\ I \in \mathcal{D}_k(I_0),\) and  \(\sum_{I \in \mathcal{D}_k(I_0)} \alpha_I =0,\) such that
\begin{equation}\label{norm1}
\bigg | \sum_{K,L \in \mathcal{D}_k({I_0})} \alpha_K \alpha_L \lambda_{KL} \bigg| = \|\Lambda\|_1.
\end{equation}

Using the symmetry of \(\Lambda\) and the fact that its row and column sums are all zero, it is easy to see that \(\|\Lambda\|_1\) is equivalent to
\[\|\Lambda\|_2 := \sup_{\substack{
                                   \|\alpha\|_{\infty} \leq 1 \\
                                   \|\beta\|_{\infty} \leq 1}}
\bigg| \sum_{K,L \in \mathcal{D}_k({I_0})} \alpha_K \beta_L \lambda_{KL} \bigg|,\]
where we take the supremum over all sequences \(\alpha=(\alpha_I)_{I \in \mathcal{D}_k(I_0)}\) and \(\beta=(\beta_I)_{I \in \mathcal{D}_k(I_0)}.\) More precisely, we have \(64 \|\Lambda\|_1 \leq \|\Lambda\|_2 \leq 192 \|\Lambda\|_1.\) Since we may assume that \(\Lambda\) is not the zero matrix (otherwise the lemma becomes trivially true), \(\|\Lambda\|_1,\) and hence \(\|\Lambda\|_2\), are not \(0.\)

If \(\alpha\) and \(\beta\) are two sequences as above, the matrix \(\Phi=(\Phi_{KL})=(\alpha_K \beta_L)\) is a rank one Schur multiplier of norm \(\|\Phi\|_m\) at most \(1\). The inequality \(\|\Lambda\|_2 \leq 192 \|\Lambda\|_1\) can be rewritten as
\[\sup \big |\langle \Phi, \Lambda \rangle \big | \leq 192 \|\Lambda\|_1,\]
where the supremum is taken over all rank one Schur multipliers \(\Phi\) of norm at most \(1\), and \(\langle \cdot, \cdot \rangle\) stands for the trace inner product.

Using Theorem \ref{grot}, we obtain that
\[\sup \big |\langle M, \Lambda \rangle \big | \leq 192 K_G \|\Lambda\|_1,\]
where the supremum is now taken over all Schur multipliers \(M\) of norm at most \(1\), and \(K_G\) is the (real) Grothendieck constant.

We have \( \sum_{K,L \in \mathcal{D}_k({I_0})} |\lambda_{KL}| = \sup \big | \langle M, \Lambda \rangle \big |,\) where the supremum is taken over all \(2^k \times 2^k\) matrices \(M\) with entries \(\pm 1\). For such a matrix \(M\) we have \(\|M\|_m \leq 2^{k/2}\), by a result of K.R. Davidson and A.P. Donsig (\cite{DaDo07}, Lemma 2.5). Putting everything together, we get the inequality
\begin{equation}\label{sumbdd}
\sum_{K,L \in \mathcal{D}_k({I_0})} |\lambda_{KL}| \leq 192 K_G 2^{k/2} \|\Lambda\|_1.
\end{equation}
Using \eqref{norm1} and \eqref{sumbdd} it follows that there exists a sequence \(\alpha\) with \(|\alpha_I| \leq 1/4\) for all \(I \in \mathcal{D}_k(I_0),\) and  \(\sum_{I \in \mathcal{D}_k(I_0)} \alpha_I =0,\) such that
\[\sum_{K,L \in \mathcal{D}_k(I_0)} |\lambda_{KL}| \leq 192 K_G 2^{k/2} \|\Lambda\|_1 =192 K_G 2^{k/2} \bigg | \sum_{K,L \in \mathcal{D}_k({I_0})} \alpha_K \alpha_L \lambda_{KL} \bigg|,\]
which is what we wanted to show. Therefore, the proof of the lemma is complete.

\end{proof}

\section{Conclusion of the proof of Theorem \ref{mainshift}}
We are now ready to finish the proof of Theorem \ref{mainshift}.

Recall that for all slices \(S_j\) of \(S\) we have
\begin{align*}
& \quad 2 \Big | \left \langle S_j f,g \right \rangle_{\lx,\lxs} \Big| \\
& \qquad \leq \sum_{L \in \mathcal{L}_j} |L|
\sum_{P,Q \in \mathcal{D}_k(L)} \bigg | \bigg \langle  \frac{\langle f \rangle _P -  \langle f \rangle _L}{2^k}, \frac{\langle g \rangle _Q -  \langle g \rangle _L}{2^k}    \bigg \rangle_{X,X^*}  +  \bigg \langle  \frac{\langle f \rangle _Q -  \langle f \rangle _L}{2^k}, \frac{\langle g \rangle _P -  \langle g \rangle _L}{2^k} \bigg \rangle_{X,X^*}   \bigg|.
\end{align*}

Let us fix \(0 \leq j \leq k-1 \) and for all \(I \in \mathcal{L}_j\) define \(A_I:=\Big (\langle f \rangle _I, \big \langle \|f\|^p_X  \big \rangle_I, \langle g \rangle _I, \big \langle \|g\|^{p'}_{X^*} \big \rangle_I \Big).\) Note that all these points are in \(\mathrm{Dom}\, \fb=\mathfrak{D}\). Lemma \ref{mainlemma} says that
\begin{align*}
& |L| \sum_{P,Q \in \mathcal{D}_k(L)} \bigg | \bigg \langle  \frac{\langle f \rangle _P -  \langle f \rangle _L}{2^k}, \frac{\langle g \rangle _Q -  \langle g \rangle _L}{2^k} \bigg \rangle_{X,X^*}  +  \bigg \langle  \frac{\langle f \rangle _Q -  \langle f \rangle _L}{2^k}, \frac{\langle g \rangle _P -  \langle g \rangle _L}{2^k}    \bigg \rangle_{X,X^*}   \bigg| \\
& \qquad \qquad \leq c 2^{k/2} \bigg( |L| \fb(A_L) - \sum_{I \in \mathcal{D}_k(L)} |I| \fb(A_I) \bigg ),
\end{align*}
for all \(L \in \mathcal{L}_j.\) We write this estimate for each \(I \in \mathcal{D}_k(L)\) and then iterate the procedure \(\ell\) times to obtain
\begin{align*}
&  \sum_{\substack{
                    I \in \mathcal{L}_j \\
                    I \subseteq L\\
                    |I| > 2^{-k \ell}|L| }}
|I| \sum_{P,Q \in \mathcal{D}_k(I)} \bigg | \bigg \langle  \frac{\langle f \rangle _P -  \langle f \rangle _I}{2^k}, \frac{\langle g \rangle _Q -  \langle g \rangle _I}{2^k}    \bigg \rangle_{X,X^*}  +  \bigg \langle  \frac{\langle f \rangle _Q -  \langle f \rangle _I}{2^k}, \frac{\langle g \rangle _P -  \langle g \rangle _I}{2^k}    \bigg \rangle_{X,X^*}   \bigg|  \\
& \qquad \qquad \qquad \leq c \cdot 2^{k/2}  \bigg( |L| \fb(A_L) - \sum_{I \in \mathcal{D}_{k \ell}(L)} |I| \fb(A_I) \bigg ) \\
& \qquad \qquad \qquad \leq c \cdot 2^{k/2} \beta_{p,X} |L| \big \langle \|f\|^p_X  \big \rangle_L^{1/p} \big \langle \|g\|^{p'}_{X^*}  \big \rangle_L^{1/{p'}} \\
& \qquad \qquad \qquad \leq c \cdot 2^{k/2} \beta_{p,X} \|f \chi_L\|_{\lx}  \|g \chi_L\|_{\lxs},
\end{align*}
where the second inequality follows from property (ii) of the Bellman function.
\par Letting \(\ell \to \infty\), we have
\begin{align*}
&  \sum_{\substack{
                    I \in \mathcal{L}_j \\
                    I \subseteq L }}
|I| \sum_{P,Q \in \mathcal{D}_k(I)} \bigg | \bigg \langle  \frac{\langle f \rangle _P -  \langle f \rangle _I}{2^k}, \frac{\langle g \rangle _Q -  \langle g \rangle _I}{2^k}    \bigg \rangle_{X,X^*}  +  \bigg \langle  \frac{\langle f \rangle _Q -  \langle f \rangle _I}{2^k}, \frac{\langle g \rangle _P -  \langle g \rangle _I}{2^k}    \bigg \rangle_{X,X^*}   \bigg|  \\
& \qquad \qquad \qquad \leq c \cdot 2^{k/2} \beta_{p,X} \|f \chi_L\|_{\lx}  \|g \chi_L\|_{\lxs}.
\end{align*}

We now cover the real line with intervals \(L \in \mathcal{L}_j\) of length \(2^M\) and apply the last inequality to each \(L\) to obtain that
\begin{align*}
&  \sum_{\substack{
                    I \in \mathcal{L}_j \\
                    |I| \leq 2^M }}
|I| \sum_{P,Q \in \mathcal{D}_k(I)} \bigg | \bigg \langle  \frac{\langle f \rangle _P -  \langle f \rangle _I}{2^k}, \frac{\langle g \rangle _Q -  \langle g \rangle _I}{2^k}    \bigg \rangle_{X,X^*}  +  \bigg \langle  \frac{\langle f \rangle _Q -  \langle f \rangle _I}{2^k}, \frac{\langle g \rangle _P -  \langle g \rangle _I}{2^k}    \bigg \rangle_{X,X^*}   \bigg|  \\
& \qquad \qquad \qquad \leq c \cdot 2^{k/2}  \beta_{p,X} \|f\|_{\lx}  \|g\|_{\lxs}.
\end{align*}

When \(M \to \infty\), we get that the norm of \(S_j\) is bounded by \(c \cdot 2^{k/2} \beta_{p,X}.\) Since \(S\) was decomposed into \(k\) slices, it follows that the operator norm of \(S\) is bounded by \(c \cdot k 2^{k/2} \beta_{p,X},\) therefore the proof of Theorem \ref{mainshift} is complete.

\bibliographystyle{plain}

\end{document}